\documentclass[12pt,reqno]{article}
\usepackage{graphicx}
\usepackage{amssymb}
\usepackage{amsmath}
\usepackage{amsthm}
\usepackage{amsfonts}
\usepackage{url}
\usepackage{hyperref}
\usepackage{breakurl}
\usepackage[T1]{fontenc}        

\newcommand{\comment}[1]{}
\newcommand{\raisecomma}{\raisebox{2pt}{$,$}}
\newcommand{\raisedot}{\raisebox{2pt}{$.$}}

\newcommand{\R}{{\mathbb R}}
\newcommand{\N}{{\mathbb N}}
\newcommand{\Npos}{{\N^{*}}}

\newcommand{\Z}{{\mathbb Z}}

\newcommand{\Pbar}{{\overline P}}
\newcommand{\Qbar}{{\overline Q}}
\newcommand{\half}{{\mbox{$\frac{1}{2}$}}} 
\newcommand{\nhalf}{{\mbox{$\frac{n}{2}$}}} 

\newtheorem{theorem}{Theorem}
\newtheorem{corollary}{Corollary}
\newtheorem{lemma}{Lemma}
\newtheorem{proposition}{Proposition}

\newtheorem{remark}{Remark}

\allowdisplaybreaks	

\begin{document}
\bibliographystyle{plain}
\title{Generalising Tuenter's binomial sums}
\author{Richard P.\ Brent\\
Mathematical Sciences Institute\\
Australian National University\\
Canberra, ACT 0200,
Australia\\
{\tt centred-sums@rpbrent.com}\\
}

\date{\today}

\maketitle
\thispagestyle{empty}                   

\pagebreak[4]

\begin{abstract}
Tuenter [\emph{Fibonacci Quarterly} 40 (2002), 175-180] and other authors
have considered centred binomial sums of the form
\[S_r(n) = \sum_k \binom{2n}{k}|n-k|^r,\]
where $r$ and $n$ are non-negative integers.
We consider sums of the form
\[U_r(n) = \sum_k \binom{n}{k}|n/2-k|^r\]
which are a generalisation of Tuenter's sums
as $S_r(n) = U_r(2n)$ but $U_r(n)$ is also
well-defined for odd arguments~$n$.
$U_r(n)$ may be interpreted as a moment of a symmetric Bernoulli random
walk with $n$ steps. 
The form of $U_r(n)$ depends on the parities of both $r$ and~$n$.
In fact, $U_r(n)$ is the product of a polynomial (depending on the parities
of~$r$ and $n$) times a power of two or a binomial coefficient. In all cases
the polynomials can be expressed in terms of Dumont-Foata polynomials. We
give recurrence relations, generating functions and explicit formulas for 
the functions $U_r(n)$ and/or the associated polynomials.
\end{abstract}

\smallskip

{\small
\noindent\emph{MSC classes:} 05A10, 11B65 (Primary);
 05A15, 05A19, 44A60, 60G50 (Secondary)
}

\pagebreak[4]
\section{Introduction}		\label{sec:intro}

We consider centred binomial sums of the form
\begin{equation}		\label{eq:U_def}
U_r(n) = \sum_k \binom{n}{k}\left|\frac{n}{2}-k\right|^r,
\end{equation}
where $r \in \N$ and $n \in \Z$.
These generalise the binomial sums
\begin{equation}		\label{eq:S_def}
S_r(n) = \sum_k \binom{2n}{k}|n-k|^r
\end{equation}
previously considered by Tuenter~\cite{Tuenter1}
and other authors~\cite{Best,rpb255,Bruckman,GZ,Hillman},
since $S_r(n) = U_r(2n)$ but $U_r(n)$ is well-defined for both even and odd
values of~$n$.
The generalisation arises naturally in the study of certain two-fold
centred binomial sums~\cite{BOOP} of the form
$\sum_j\sum_k\binom{2n}{n+j}\binom{2n}{n+k}P(j,k)$.

In definitions such as~\eqref{eq:U_def} and~\eqref{eq:S_def} we always
interpret $0^0$ as $1$. 
Thus $U_0(n) = 2^n$ and $S_0(n) = 2^{2n}$ for all $n\in \N$.
By our summation convention (see~\S\ref{subsec:notation}),
we have $U_r(n) = S_r(n) = 0$ if $n < 0$. Thus, in the following
we assume that $n\ge 0$.

For $r > 0$ we can avoid the absolute value function
in~\eqref{eq:U_def} by writing
\[U_r(n) = 2\sum_{k < n/2} \binom{n}{k}\left(\frac{n}{2}-k\right)^r.\]

Tuenter~\cite{Tuenter1}
showed in a direct manner that, for $r \ge 0$ and $n > 0$, $S_r(n)$
satisfies the recurrence
\begin{equation}	\label{eq:Tuenter_rec0}
S_{r+2}(n) = n^2 S_r(n) - 2n(2n-1)S_r(n-1).
\end{equation}
Observe that this recurrence splits into two separate recurrences, one
involving odd values of~$r$ and the other involving even values of~$r$.
Also, 
$S_0(n) = 2^{2n}$ and $S_1(n) = n\binom{2n}{n}$ (see for example
\cite{rpb255,Hillman}).
It follows from~\eqref{eq:Tuenter_rec0} that
\begin{equation}		\label{eq:SPQ}
S_{2r}(n) = Q_r(n)2^{2n-r}, \;\; S_{2r+1}(n) = P_r(n)n\binom{2n}{n},
\end{equation}
where $P_r(n)$ and $Q_r(n)$ are polynomials of degree~$r$ with integer
coefficients, satisfying the recurrences
\begin{align}	\label{eq:Prec}
P_{r+1}(n) &= n^2P_r(n) - n(n-1)P_r(n-1),\\
		\label{eq:Qrec}
Q_{r+1}(n) &= 2n^2Q_r(n) - n(2n-1)Q_r(n-1)
\end{align}
for $r \ge 0$,
with initial conditions $P_0(n) = Q_0(n) = 1$.
The polynomials $P_r, Q_r$ for $0 \le r \le 5$ are given in Appendix~$1$.

The Dumont-Foata polynomials $F_r(x,y,z)$ are $3$-variable polynomials
satisfying the recurrence relation
\begin{equation}		\label{eq:DF-recurrence}
F_{r+1}(x,y,z) = (x+z)(y+z)F_r(x,y,z+1) - z^2F_r(x,y,z)
\end{equation}
for $r \ge 1$, with $F_1(x,y,z) = 1$.
Dumont and Foata~\cite{DF} gave a combinatorial interpretation for the
coefficients of $F_r(x,y,z)$ and showed that $F_r(x,y,z)$ is symmetric in the
three variables $x,y,z$.

Tuenter~\cite{Tuenter1} showed that $P_r(n)$ and $Q_r(n)$ may be expressed
in terms of Dumont-Foata polynomials.
In fact, for $r \ge 1$,
$P_r(n) = (-)^{r-1}nF_r(1,1,-n)$
and 
$Q_r(n) = (-2)^{r-1}nF_r(\half,1,-n)$.
Thus, we can obtain explicit formulas and generating functions for
the polynomials $P_r(n)$ and $Q_r(n)$ as special cases of the
results of Carlitz~\cite{Carlitz} on Dumont-Foata polynomials.

We can obtain explicit formulas for $S_{2r}(n)$ and $S_{2r-1}(n)$ by
using Carlitz's results for Dumont-Foata polynomials~--
see Theorem~\ref{thm:explicit} in~\S\ref{sec:main}\footnote{In fact,
our Theorem~\ref{thm:explicit} is more general,
since it covers $U_r(n)$ for both even and odd~$n$.
}.
We note that these formulas are different from
the explicit formulas~\eqref{eq:GZ7.1}--\eqref{eq:GZ3.2} of
Guo and Zeng~\cite{GZ}, which are discussed in
Remark~\ref{remark:other_explicit}.

We show that all the above results for $S_r(n)$ can be generalised
to cover $U_r(n)$. In particular, Theorem~\ref{thm:recurrence} shows that
$U_r(n)$ satisfies a recurrence~\eqref{eq:Urec}
similar to the recurrence~\eqref{eq:Tuenter_rec0} satisfied by $S_r(n)$.
Theorem~\ref{thm:polynomials} shows that
$U_r(n)$ is the product
of a polynomial in~$n$ times a power of two or a binomial coefficient,
depending on the parity of $r$, as in~\eqref{eq:SPQ}.
These polynomials can be expressed in terms of Dumont-Foata polynomials,
so the results of Carlitz allow us to obtain explicit formulas
for $U_r(n)$ such as those given in Theorem~\ref{thm:explicit},
and to obtain new exponential generating functions (\emph{egfs})
such as~\eqref{eq:U_even_gf}
and~\eqref{eq:U_odd_even_egf}--\eqref{eq:U_odd_odd_egf}
in~\S\ref{sec:egfs}.
We give some additional explicit formulas in~\S\ref{sec:explicit}, and
consider the asymptotic behaviour of $U_r(n)$ as $n\to\infty$
in~\S\ref{sec:asymptotics}.

\subsection*{Acknowledgement}

We thank Hideyuki Ohtsuka for informing us of the 
paper 
\cite{GZ}.
This research was supported by Australian Research Council grant DP140101417.

\pagebreak[3]

\subsection{Notation}		\label{subsec:notation}

The set of non-negative integers is denoted by $\N$, and the
set of positive integers by $\Npos$. 

For $k \in \N$ and $x\in\R$
we denote the \emph{Pochhammer symbol}
or \emph{rising factorial}
by
\[(x)_k := x(x+1)\cdots(x+k-1),\]
with the special case $(x)_0 = 1$. 
The falling factorial may be written as $(x+1-k)_k$ or 
$(-)^k(-x)_k$, where we use $(-)^k$ as an abbreviation for $(-1)^k$.

The binomial coefficient $\binom{n}{k}$ is defined\footnote{Guo
and Zeng~\cite{GZ} implicitly define the binomial coefficient
to be zero if $-k \in \Npos$, $n\in\Z$, and 
$n(n-1)\cdots(n-k+1)/k!$ if $k \in \N$, $n\in\Z$.
We do not use this definition because it is
incompatible with our convention of summing over
all $k\in\Z$.}
for all $k \in \Z$ and $n \in \N$ by
\[\binom{n}{k} := \begin{cases}
		  0 \text{ if } k < 0 \text{ or } k > n;\\
	 	  \frac{n!}{(n-k)!\,k!} \text{ otherwise.}
		  \end{cases}
\]
Thus we can often write sums over all $k\in\Z$ 
without explicitly giving upper and lower limits on~$k$.

\section{Main results}		\label{sec:main}

Our main results on $U_r(n)$ are summarised in the following
Theorems~\ref{thm:recurrence}--\ref{thm:explicit}.
The recurrence~\eqref{eq:Urec} in
Theorem~\ref{thm:recurrence} implies
the recurrence~\eqref{eq:Tuenter_rec0} satisfied by
$S_r(n)$, since~\eqref{eq:Tuenter_rec0}
follows on replacing $n$ by $2n$ in~\eqref{eq:Urec}.

\begin{theorem}		\label{thm:recurrence}
For all $r, n \in \N$, $U_r(n)$ satisfies the recurrence
\begin{equation}			\label{eq:Urec}
4U_{r+2}(n) = n^2U_r(n) - 4n(n-1)U_r(n-2),
\end{equation}
and  may be computed from the recurrence using the initial
conditions
\[
U_0(n) = 2^n,\;
U_1(2n) = n\binom{2n}{n},\;
U_1(2n+1) = (2n+1)\binom{2n}{n}
\text{ for all } n \in \N.
\]
\end{theorem}

\begin{proof}
We have
\begin{align}
4U_{r+2}(n) &= \sum_k 4\binom{n}{k}\left|\frac{n}{2}-k\right|^{r+2}
	= \sum_k \binom{n}{k}\left|\frac{n}{2}-k\right|^{r}(n-2k)^2,
	\label{eq:rec1}
\end{align}
\begin{align}
n^2U_r(n) &= \sum_k\binom{n}{k}\left|\frac{n}{2}-k\right|^{r}n^2,
	\label{eq:rec2}
\end{align}
and
\begin{align}
4n(n-1)U_r(n-2) &= 
 \sum_k 4n(n-1)\binom{n-2}{k}\left|\frac{n-2}{2}-k\right|^r
	\nonumber\\
 &= \sum_k 4n(n-1)\binom{n-2}{k-1}\left|\frac{n}{2}-k\right|^{r}
	\nonumber\\
 &= \sum_k 4k(n-k)\binom{n}{k}\left|\frac{n}{2}-k\right|^{r}.
	\label{eq:rec3}
\end{align}
Since $(n-2k)^2 - n^2 - 4k(n-k) = 0$, the recurrence~\eqref{eq:Urec}
follows from~\eqref{eq:rec1}--\eqref{eq:rec3}.

For the initial values, we easily verify that $U_0(n) = 2^n$.
Also, 
the ``official'' solution~\cite{Hillman} to the Putnam problem 35-A4
gives
\begin{align*}
U_1(n) 
 &= \sum_k\binom{n}{k}\left|\frac{n}{2}-k\right|
  = \sum_{k < n/2}\binom{n}{k}(n-2k)\\
 &= \sum_{k < n/2}\left\{\binom{n}{k}(n-k) - \binom{n}{k}k\right\}\\
 &= \sum_{k < n/2}\left\{\binom{n-1}{k}n - \binom{n-1}{k-1}n\right\}\\
 &= n\sum_{k < n/2}\left\{\binom{n-1}{k}-\binom{n-1}{k-1}\right\}\\
 &= n\binom{n-1}{\lfloor n/2 \rfloor}.
\end{align*}
Thus, 
\[U_1(2n) = 2n\binom{2n-1}{n} = n\binom{2n}{n}\]
and
\[U_1(2n+1) = (2n+1)\binom{2n}{\lfloor(2n+1)/2\rfloor}
	    = (2n+1)\binom{2n}{n}.\]
\end{proof}

Theorem~\ref{thm:polynomials} shows that
$U_r(n)$ can be expressed as the product of a polynomial in~$n$
multiplied by a simple non-polynomial function of~$r$ and~$n$. 
There are four cases, depending on the parities
of~$r$ and~$n$, although only three of the cases
are essentially different.

\begin{theorem}		\label{thm:polynomials}
For $r\in\N$ there
exist polynomials $P_r(n), \Pbar_r(n), Q_r(n), \Qbar_r(n)$
of degree~$r$ over $\Z$, such that, for all $n\in\Npos$,
\begin{align}
U_{2r+1}(2n) &= nP_r(n)\binom{2n}{n},
	\label{eq:U_odd_even_P}\\
U_{2r+1}(2n-1) &= 2^{-(2r+1)}n\Pbar_r(n)\binom{2n}{n},
	\label{eq:U_odd_odd_Pbar}\\
U_{2r}(2n) &= 2^{2n-r}Q_r(n),
	\label{eq:U_even_even_Q}\\
U_{2r}(2n+1) &= 2^{2n+1-2r}\Qbar_r(n).
	\label{eq:U_even_odd_Qbar}
\end{align}
The polynomials satisfy the following recurrence relations:
\begin{align}
P_{r+1}(n) &= n^2 P_r(n) -n(n-1)P_r(n-1),
 \label{eq:P_rec}\\
\Pbar_{r+1}(n) &= (2n-1)^2\,\Pbar_r(n) - 4(n-1)^2\,\Pbar_r(n-1),
 \label{eq:Pbar_rec}\\
Q_{r+1}(n) &= 2n^2 Q_r(n) - n(2n-1)Q_r(n-1),
 \label{eq:Q_rec}\\
\Qbar_{r+1}(n) &= (2n+1)^2\,\Qbar_r(n) - 2n(2n+1)\Qbar_r(n-1),
 \label{eq:Qbar_rec}
\end{align}
with initial conditions
$P_0(n) = \Pbar_0(n) = Q_0(n) = \Qbar_0(n) = 1$.
\end{theorem}

\begin{proof}
For $r\in\N$ and $n \in \Npos$ we \emph{define} functions
$P_r(n),\Pbar_r(n),Q_r(n),\Qbar_r(n)$ by
\eqref{eq:U_odd_even_P}--\eqref{eq:U_even_odd_Qbar} respectively.
Using the second half of Theorem~\ref{thm:recurrence}, it is easy to
see that $P_0(n) = \Pbar_0(n) = Q_0(n) = \Qbar_0(n) = 1$. Thus, it only
remains to show that $P_r(n),\Pbar_r(n),Q_r(n)$ and $\Qbar_r(n)$ satisfy
the claimed recurrences~\eqref{eq:P_rec}--\eqref{eq:Qbar_rec},
since these recurrences enable us to show by induction on~$r$ that
the functions $P_r(n),\Pbar_r(n),Q_r(n)$ and $\Qbar_r(n)$ are
polynomials over~$\Z$.

First consider the recurrence~\eqref{eq:P_rec} for $P_r(n)$.
Replacing $n$ by $2n$ and $r$ by $2r+1$ in the recurrence~\eqref{eq:Urec},
we obtain
\[U_{2r+3}(2n) = n^2U_{2r+1}(2n)-2n(2n-1)U_{2r+1}(2n-2),\]
and in view of~\eqref{eq:U_odd_even_P} this implies
\begin{equation}	\label{eq:Prec_unsimplified}
nP_{r+1}(n)\binom{2n}{n} = 
  n^2P_r(n)\binom{2n}{n} - 2n(2n-1)(n-1)P_r(n-1)\binom{2n-2}{n-1}.
\end{equation}
Now, dividing each side of~\eqref{eq:Prec_unsimplified} by $n\binom{2n}{n}$
and using
$\binom{2n-2}{n-1} = \frac{n}{2(2n-1)}\binom{2n}{n}$, we
obtain~\eqref{eq:P_rec}.  The other three cases
are similar.
\end{proof}

Lemma~\ref{lemma:DF} expresses the four families of polynomials
$P_r(n),\ldots,\Qbar_r(n)$ in terms of Dumont-Foata polynomials,
and incidentally shows that only three of the four cases are
essentially different, since $\Qbar_r(n)$ is just a shifted and
scaled version of $Q_r(n)$.
\begin{lemma}		\label{lemma:DF}
For $r\in\Npos$, the polynomials of Theorem~$\ref{thm:polynomials}$ can
be expressed in terms of Dumont-Foata polynomials, as follows:
\begin{align}
P_r(n) &= (-)^{r-1} n F_r(-n,1,1),	\label{eq:P_F}\\
\Pbar_r(n) &= (-4)^r F_{r+1}\left(\half-n,\half,\half\right),
					\label{eq:Pbar_F}\\
Q_r(n) &= (-2)^{r-1} n F_r\left(-n,\half,1\right),
					\label{eq:Q_F}\\
\Qbar_r(n) &=
(-)^{r-1}2^{2r-1} \left(n+\half\right)F_r\left(-n-\half,\half,1\right) =
2^r Q_r\left(n+\half\right). 		\label{eq:Qbar_and_Q}
\end{align}
\end{lemma}

\begin{proof}
Since $F_r(x,y,z)$ is undefined for $r \le 0$, we assume that $r\in\Npos$.
From the defining recurrence~\eqref{eq:DF-recurrence} it is easy to
verify that the function $f_r(n) := (-)^{r-1}nF_r(-n,1,1)$ satisfies
the recurrence~\eqref{eq:P_rec} that is satisfied by $P_r(n)$.  Also,
$f_1(n) = nF_1(-n,1,1) = n$, so $f_1(n) = P_1(n)$.  The
recurrence~\eqref{eq:P_rec} uniquely defines the polynomial $P_r(n)$,
so $P_r(n) = f_r(n)$ and~\eqref{eq:P_F} holds.
The other three cases
are similar.
\end{proof}

Using Lemma~\ref{lemma:DF},
we obtain $U_r(n)$ in terms of Dumont-Foata polynomials.
There are three cases, depending on the parities of $r$ and $n$: (even,
any), (odd, even) and (odd, odd).

\begin{corollary}	\label{cor:DF}
\begin{align}
\label{eq:U_even_a}
U_{2r}(n) &= 2^{n-2}\, n\, (-)^{r-1} F_r\left(-\nhalf,\half,1\right),\\
\label{eq:U_odd_even}
U_{2r+1}(2n) &= n^2\, (-)^{r-1} F_r(-n,1,1) \binom{2n}{n},\\
\label{eq:U_odd_odd}
U_{2r+1}(2n-1) &= \half n(-)^r F_{r+1}\left(\half-n,\half,\half\right)
		\binom{2n}{n}.
\end{align}
The expressions
\eqref{eq:U_even_a}--\eqref{eq:U_odd_even} are valid for $r \ge 1$,
and \eqref{eq:U_odd_odd} is valid for $r \ge 0$.
\end{corollary}

\begin{proof}
This is immediate from \eqref{eq:U_odd_even_P}--\eqref{eq:U_even_odd_Qbar}
of Theorem~\ref{thm:polynomials} and \eqref{eq:P_F}--\eqref{eq:Qbar_and_Q}
of Lemma~\ref{lemma:DF}.
\end{proof}

Proposition~\ref{prop:DF}, due to Carlitz~\cite{Carlitz}, gives
an explicit formula for the Dumont-Foata polynomials.
\begin{proposition}	\label{prop:DF}
For $r \in \Npos$ we have
\begin{equation}			\label{eq:Carlitz_thm1}
F_r(x,y,z) = 2(-)^{r-1}\sum_{0 \le j \le k < r}
	(-)^j\frac{(x+z)_k(y+z)_k(z+j)^{2r-1}}{j!(k-j)!(2z+j)_{k+1}}
	\,\raisecomma
\end{equation}
provided that the Pochhammer symbol
in the denominator does not vanish, which is equivalent
to saying that $(2z)_{2r-1} \ne 0$.
In particular, \eqref{eq:Carlitz_thm1} is valid
for all positive $z$.
\end{proposition}
\begin{proof}
Modulo a small change of notation, this is
Carlitz~\cite[Thm.~1]{Carlitz}\footnote{Carlitz does not
state the condition $(2z)_{2r-1} \ne 0$. In fact, if we evaluate
the RHS of~\eqref{eq:Carlitz_thm1} symbolically and cancel any factors
of the form $2z+j$ that occur in both the numerator and the denominator,
then the result is always valid.}.

\end{proof}

Theorem~\ref{thm:explicit} gives explicit formulas for $U_r(n)$. 
The definition~\eqref{eq:U_def} can be used to evaluate
$U_r(n)$ for small~$n$, but this is infeasible if $n$ is large, as there
are $n+1$ terms in the defining sum. Hence it is preferable to
use the appropriate explicit formula of Theorem~\ref{thm:explicit}
when $r$ is small but $n$ is large\footnote{The
recurrence relations~\eqref{eq:Urec}
or~\eqref{eq:P_rec}--\eqref{eq:Qbar_rec}
may also be used in this case.}.

\pagebreak[3]

\begin{theorem}			\label{thm:explicit}
For $r, n \in \Npos$ we have the following explicit formulas:
\begin{equation}		\label{eq:U_even_a2}
U_{2r}(n) = 2^{n+1} \sum_{1 \le j \le k \le r}
	(-)^j \frac{\left(-\frac{n}{2}\right)_k 
	    \left(\half\right)_k}
		    {(k-j)!(k+j)!}\,j^{2r},
\end{equation}
\begin{equation}	\label{eq:U_odd_even2}
U_{2r+1}(2n) = 2n\binom{2n}{n}\sum_{1 \le j \le k \le r}
	(-)^{j}\frac{(-n)_{k}}{(k-j)!\,(k+1)_j}\,j^{2r},
\end{equation}
\begin{equation}	\label{eq:U_odd_odd2}
U_{2r-1}(2n-1) = \binom{2n}{n}\sum_{1 \le j \le k \le r}
	(-)^{j}\frac{(-n)_{k}}{(k-j)!\,(k)_j}\,
	  \left(j-\half\right)^{2r-1}.
\end{equation}
\end{theorem}

\begin{proof}

We can substitute
the identity~\eqref{eq:Carlitz_thm1}
into each of~\eqref{eq:U_even_a}--\eqref{eq:U_odd_odd} to obtain
explicit formulas for the different cases of $U_r(n)$.
For example, \eqref{eq:U_even_a} gives~\eqref{eq:U_even_a2}.
Similarly, \eqref{eq:U_odd_even} gives~\eqref{eq:U_odd_even2},
and \eqref{eq:U_odd_odd} gives
\eqref{eq:U_odd_odd2} on replacing $r$ by $r-1$.
\end{proof}

\vspace*{\fill}\pagebreak[3]

\begin{remark}
{\rm
Replacing $F_r(-\nhalf,\half,1)$ by $F_r(-\nhalf,1,\half)$
before applying \eqref{eq:Carlitz_thm1} 
gives
\begin{equation}		\label{eq:U_even_b2}
U_{2r}(n) = 2^n n \sum_{1 \le j \le k \le r}
	(-)^{j-1} \frac{\left(\frac{1-n}{2}\right)_{k-1} 
           \left(\half\right)_k}
 	     {(k+j-1)!(k-j)!}\,\left(j-\half\right)^{2r-1}.
\end{equation}
We note that formulas~\eqref{eq:U_even_a2} and~\eqref{eq:U_even_b2}
are different.
}
\end{remark}

\begin{remark}
{\rm
Other explicit expressions for $U_r(n)$ may be obtained by applying the
above process after interchanging
$x$ and $z$ in~\eqref{eq:Carlitz_thm1}, but these expressions
require the condition $n \ge r$ to avoid division by zero due to 
the vanishing of a Pochhammer symbol in the denominator.
Hence, we omit the details. 
}
\end{remark}

\section{Explicit formulas via Lagrange interpolation} \label{sec:explicit}

We can obtain explicit formulas via polynomial
interpolation whenever the function being interpolated is a polynomial
of known degree.  For example,
from~\eqref{eq:U_even_even_Q}--\eqref{eq:U_even_odd_Qbar} and
\eqref{eq:Qbar_and_Q} we have
$U_{2r}(n) = 2^{n-r}Q_r(n/2)$, where $Q_r(x)$ is a polynomial of degree~$r$.
Thus, we can obtain an explicit formula for $Q_r(x)$, and hence for
$U_{2r}(n)$, by evaluating $Q_r(x)$ at $r+1$ distinct points $x_k$ and using
Lagrange's polynomial interpolation formula
\[Q_r(x) = \sum_{k=0}^r Q_r(x_k)L_k(x),\]
where
\[L_k(x) = \prod_{\substack{0\le j \le r \\ j\ne k}} 
		\frac{x-x_j}{x_k-x_j}\,\raisedot\]
For example, taking $x_k = k$ for $0 \le k \le r$ gives the explicit
formula
\begin{equation}	\label{eq:U_even_any_Lagrange}
U_{2r}(n) = (-)^r\, 2^{n+1}\! \!\sum_{1 \le j \le k \le r}\!
	\frac{\left(-\nhalf\right)_k \left(\half\right)_k 
		\left(\nhalf-r\right)_{r-k}}
		{(k+j)! (k-j)! (r-k)!}\,j^{2r},
\end{equation}
which is valid for $n\in\Npos$,
and may be compared with~\eqref{eq:U_even_a2}
and~\eqref{eq:U_even_b2}.

In the same way, by interpolating $P_r(n)$ at $n=0,1,\dots,r$,
using $P_r(0) = 0$ for $r \ge 1$,
we obtain the explicit formula
\begin{equation}	\label{eq:U_odd_even_Lagrange}
U_{2r+1}(2n) = 2n\binom{2n}{n}(-)^r \!\sum_{1 \le j \le k \le r}\!
	\frac{(-n)_k(n-r)_{r-k}}{(r-k)!(k-j)!(k)_{j+1}}
	\,j^{2r+1},
\end{equation}
which is valid for $r\in\Npos$ and $n\in\N$,
and may be compared with~\eqref{eq:U_odd_even2}.

Similarly, by interpolating $n\Pbar_{r-1}(n)$ at $n = 0, 1, \ldots, r$,
we obtain
\begin{equation}	\label{eq:U_odd_odd_Lagrange}
U_{2r-1}(2n-1) = (-)^r\binom{2n}{n}\!\sum_{0 \le j \le k \le r}\!
	\frac{(1-n)_k (n-1-r)_{r-k}}{(r-k)!(k-j)!(k+2)_j}
	\,\left(j+\half\right)^{2r-1}
\end{equation}
for $r \ge 1$ and $n \ge 1$.  
This may be compared with~\eqref{eq:U_odd_odd2}.

We can obtain an infinite number of explicit formulas by evaluating
the relevant polynomials at different sets of $r+1$ distinct points.
The examples given above seem the most natural.

\begin{remark}		\label{remark:other_explicit}
{\rm
Recently, Guo and Zeng~\cite{GZ} considered $S_r(n)$.
They obtained explicit formulas which may be written,
for $r \ge 1$, as
\begin{equation}		\label{eq:GZ7.1}
S_{2r}(n) = \sum_{0 \le j \le k < r}
   2^{2n-2k-1}(-)^{k-j}\binom{2n}{j}
     \frac{(2n-2k)_{k-j}}{(k-j)!}\,
       (n-j)^{2r-1}
\end{equation}
and 
\begin{equation}		\label{eq:GZ3.2}
S_{2r-1}(n) = \!\!\!\!\sum_{0 \le j \le k \le \min(r-1,n)}\!\!\!\!\!\!
  (-)^{k-j}
    \binom{2n-2k}{n-k}\binom{2n}{j}
       \frac{(2n-2k)_{k-j}}{(k-j)!}\,
          (n-j)^{2r-1}.
\end{equation}

Assuming that $n \ge r$,
the explicit formulas~\eqref{eq:GZ7.1}--\eqref{eq:GZ3.2}
may be obtained by interpolation of
the polynomials $Q_r(n)/n$ and $P_{r-1}(n)$ respectively,
using the $r$ points
$n, n-1, \ldots, n-r+1$. 
Similar remarks apply to the results of
Chen and Chu~\cite[Thm.~1~--~Cor.~4]{CC}.
}
\end{remark}

\section{Generating functions}		\label{sec:egfs}

Theorem~\ref{thm:Ugen_even_any} gives an egf~\eqref{eq:U_even_gf}
for $U_{2r}(n)$ and any fixed $n\in\N$.  It generalises the egf
\begin{equation}			\label{eq:S_even_gf}
\sum_{r\ge 0} S_{2r}(n)\frac{z^{2r}}{(2r)!}
 = 2^{2n} \cosh^{2n} (z/2)
\end{equation}
given by Tuenter~\cite[\S5]{Tuenter1},
since replacing $n$ by $2n$ in~\eqref{eq:U_even_gf}
gives~\eqref{eq:S_even_gf}. The proof is straightforward,
and does not require the results of Carlitz.
\begin{theorem}		\label{thm:Ugen_even_any}
For $n\in\N$ we have the exponential generating function
\begin{equation}			\label{eq:U_even_gf}
\sum_{r\ge 0} U_{2r}(n)\frac{z^{2r}}{(2r)!}
 = 2^n \cosh^n (z/2).
\end{equation}
\end{theorem}
\begin{proof}
From the definition~\eqref{eq:U_def} with $r$ replaced by $2r$ (so the
absolute value signs can be omitted), we have
\begin{align*}
\sum_{r\ge 0}U_{2r}(n)\frac{z^{2r}}{(2r)!} 
 &= \sum_{r\ge 0}\sum_k\binom{n}{k}\left(\frac{n}{2}-k\right)^{2r}
	\frac{z^{2r}}{(2r)!}\\
 &= \sum_k\binom{n}{k}\sum_{r\ge 0}\left(\frac{n}{2}-k\right)^{2r}
	\frac{z^{2r}}{(2r)!}\\
 &= \sum_k\binom{n}{k}\cosh\left(\left(\frac{n}{2}-k\right)z\right)\\
 &= \frac{1}{2}\sum_k\binom{n}{k}
	\left(e^{(n/2-k)z} + e^{(k-n/2)z}\right)\\
 &= \sum_k\binom{n}{k}e^{(k-n/2)z}\\
 &= e^{-nz/2}\left(1+e^z\right)^n\\
 &= \left(e^{-z/2}+e^{z/2}\right)^n = 2^n\cosh^n(z/2),
\end{align*}
as required.
The series
converges absolutely for all~$z$.
\end{proof}

We can obtain other egfs from the results of
Carlitz.
First, we note that Carlitz~\cite[eqn.~(4.2)]{Carlitz} gives the egf
\begin{equation}		\label{eq:Carlitz42}
\sum_{r \ge 1} (-)^rF_r(x,y,1) \frac{z^{2r}}{(2r)!}
 = \frac{1}{xy}\sum_{k\ge 1}(-)^k\frac{(x)_k(y)_k}{(2k)!}
	\left(2\sinh\frac{z}{2}\right)^{2k}.
\end{equation}
In view of~\eqref{eq:U_even_a} and~\eqref{eq:U_odd_even}, this 
allows us to obtain egfs for $U_{2r}(n)$
and $U_{2r+1}(2n)$.
{From}~\eqref{eq:U_even_a} and~\eqref{eq:Carlitz42} we obtain
\begin{equation}		\label{eq:U_even_both_egf}
\sum_{r\ge 0}U_{2r}(n)\frac{z^{2r}}{(2r)!} =
	2^n\,\sum_{k \ge 0}(-)^k
	\left(\half\right)_k\left(-\nhalf\right)_k
	\frac{\left(2\sinh(z/2)\right)^{2k}}{(2k)!}\,\raisedot
\end{equation}
Comparing this with~\eqref{eq:U_even_gf}, we must have
\begin{equation}		\label{eq:sinh_cosh}
\cosh^n(z) = \sum_{k \ge 0}(-)^k
        \left(\half\right)_k\left(-\nhalf\right)_k
        \frac{\left(2\sinh(z)\right)^{2k}}{(2k)!}\,\raisedot
\end{equation}

\pagebreak[3]
\noindent
Indeed, if we define $s := \sinh^2 z$, so $1+s = \cosh^2 z$, then
the left side of~\eqref{eq:sinh_cosh} is $(1+s)^{n/2}$,
and the right side is a disguised form of the binomial expansion
$(1+s)^{n/2} = 1 + (n/2)s + (n/2)(n/2-1)s^2/2! + \cdots$.
Thus,~\eqref{eq:U_even_a} and~\eqref{eq:Carlitz42}
give nothing new; they merely confirm~\eqref{eq:U_even_gf}.

To obtain something new, we
consider~\eqref{eq:U_odd_even} and~\eqref{eq:Carlitz42}.
Proceeding as above, we obtain the interesting egf:
\begin{equation}	\label{eq:U_odd_even_egf}
\sum_{r\ge 0}U_{2r+1}(2n)\frac{z^{2r}}{(2r)!} =
 n\binom{2n}{n}\sum_{k=0}^n 2^{2k}\binom{n}{k}\binom{2k}{k}^{-1}
  \sinh^{2k}\left(\frac{z}{2}\right)\,\raisedot
\end{equation}
Observe that, in order to calculate $U_{2r+1}(2n)$
from~\eqref{eq:U_odd_even_egf}, it is only necessary to sum
the terms on the right-hand side 
for $k \le \min(r,n)$.

The final case $U_{2r+1}(2n-1)$ is more difficult because
\eqref{eq:Carlitz42} does not apply to~\eqref{eq:U_odd_odd},
as the last argument of $F_{r+1}(x,y,z)$ in~\eqref{eq:U_odd_odd}
is $\half$, not $1$.
However, we can use the egf
\begin{align}
\sum_{r\ge 0}&(-)^r F_{r+1}(x,y,z)\frac{u^{2r+1}}{(2r+1)!} \nonumber\\
 &=\, 2\,\sum_{k=0}^\infty \sum_{j=0}^k
     \frac{(-)^j (x+z)_k (y+z)_k (2z)_j}{j! (k-j)! (2z)_{k+1} (2z+k+1)_j}
        \sinh((z+j)u),	\label{eq:Carlitz222}
\end{align}
which follows\footnote{Carlitz does not give~\eqref{eq:Carlitz222}
explicitly.  He gives a generating function
\cite[eqn.~(4.6)]{Carlitz} involving the hypergeometric function.
However, there is a problem with convergence of the Maclaurin series
involved, because the hypergeometric function occurs with
arguments $e^u$ and $e^{-u}$, one of which lies outside the unit circle.
Thus, Carlitz's generating function is only valid (if at all)
in the context of formal power series.
We prefer to use~\eqref{eq:Carlitz222}, for which there is no problem
with convergence.
}
from the discussion in Carlitz~\cite[pp.~221--222]{Carlitz}.
Using~\eqref{eq:U_odd_odd}
and~\eqref{eq:Carlitz222}, after some simplification followed by
a change of variables ($u \mapsto z$),
we obtain the following egf:

\begin{align}
\sum_{r\ge 0}&\,U_{2r+1}(2n-1)\frac{z^{2r+1}}{(2r+1)!} \nonumber \\
&=\,
n\binom{2n}{n}\sum_{0 \le j \le k < n}
	\frac{(-)^{k-j}\binom{n-1}{k}\binom{2k}{k-j}}{\binom{2k}{k}}\,
	\frac{\sinh\left(\left(j+\half\right)z\right)}
		{j+k+1}\,\raisecomma
	\label{eq:U_odd_odd_egf}
\end{align}
which is valid for $n\in\Npos$.

\section{Asymptotics}		\label{sec:asymptotics}

Using Stirling's formula to approximate the binomial coefficient
in~\eqref{eq:U_def}, we see that $U_r(n)$ can be regarded as a Riemann
sum approximating a suitably scaled integral of the form
\[\int_{-\infty}^{\infty} e^{-x^2}|x|^r\,{\rm d}x.\]
This gives an asymptotic approximation to $U_r(n)$ as $n \to \infty$
with $r$ fixed.  More precisely,
if the notation $O_r(\cdots)$ means that the implied
constant \hbox{depends} on~$r$, we have
\begin{equation}		\label{eq:asymptotics}
U_r(n) = \pi^{-1/2}\,2^n \left(\frac{n}{2}\right)^{r/2}
		\Gamma\left(\frac{r+1}{2}\right)
		\left(1 + O_r\left(\frac{1}{n}\right)\right)
\end{equation}
as $n \to \infty$.  
The same asymptotic approximation
can be obtained by ignoring all but the leading
coefficient in the polynomials $P_r(n), \cdots, \Qbar_r(n)$.
(The leading coefficients are given in Table~$1$ of Appendix~$2$.)
For example, \[P_r(n) = r!n^r + O_r(n^{r-1}),\] so~\eqref{eq:U_odd_even_P}
and the approximation 
\[\binom{2n}{n} = \frac{2^{2n}}{\sqrt{\pi n}}
		\left(1 + O\left(\frac{1}{n}\right)\right)\]
give
\[U_{2r+1}(2n) = \pi^{-1/2}2^{2n}n^{r+1/2}\,r!
		\left(1 + O_r\left(\frac{1}{n}\right)\right),\]
which is a special case of~\eqref{eq:asymptotics}. 
Other special cases of~\eqref{eq:asymptotics} can be obtained from
\eqref{eq:U_odd_odd_Pbar}--\eqref{eq:U_even_odd_Qbar}.

\vspace*{\fill}\pagebreak[4]

\pagebreak[4]

\subsection*{Appendix 1: Small cases of the polynomials $P, Q, \Pbar, \Qbar$}

\begin{align*}
P_0(n) &= 1,\\
P_1(n) &= n,\\
P_2(n) &= n(2n-1),\\
P_3(n) &= n(6n^2 - 8n + 3),\\
P_4(n) &= n(24n^3 - 60n^2 + 54n - 17),\\
P_5(n) &= n(120n^4 - 480n^3 + 762n^2 - 556n + 155),\\
Q_0(n) &= 1,\\
Q_1(n) &= n,\\
Q_2(n) &= n(3n-1),\\
Q_3(n) &= n(15n^2 - 15n + 4),\\
Q_4(n) &= n(105n^3 - 210n^2 + 147n - 34),\\
Q_5(n) &= n(945n^4 - 3150n^3 + 4095n^2 - 2370n + 496),\\
\Pbar_0(n) &= 1,\\
\Pbar_1(n) &= 4n-3,\\
\Pbar_2(n) &= 32n^2 - 56n + 25,\\
\Pbar_3(n) &= 384n^3 - 1184n^2 + 1228n - 427,\\
\Pbar_4(n) &= 6144n^4 - 29184n^3 + 52416n^2 - 41840n + 12465,\\
\Pbar_5(n) &= 122880n^5 - 829440n^4 +\\
	   & \phantom{= xx} 2258688n^3 - 3076288n^2
		+ 2079892n - 555731,\\
\Qbar_0(n) &= 1,\\
\Qbar_1(n) &= 2n+1,\\
\Qbar_2(n) &= 12n^2 + 8n + 1,\\
\Qbar_3(n) &= 120n^3 + 60n^2 + 2n + 1,\\
\Qbar_4(n) &= 1680n^4 - 168n^2 + 128n + 1,\\ 
\Qbar_5(n) &= 30240n^5 - 25200n^4 + 5040n^3 + 7320n^2 - 2638n + 1.
\end{align*}

The triangles of coefficients of $-P_r(-n)/n, -Q_r(-n)/n$ and
$\Qbar_r(-n)$ for $r \ge 1$ are OEIS~\cite{OEIS} 
sequences A036970, A083061 and A160485 respectively. 
We have contributed the coefficients of $\Pbar_r(n)$
as sequence A245244.
The values $(-)^r\Pbar_r(0)$ are
sequence A009843 (see Appendix~$2$ for details).

The bijection~\eqref{eq:Qbar_and_Q} between A083061 and A160485
(by a shift of $\pm\half$ and scaling by a power of~$2$)
was not mentioned in the relevant OEIS entries as at July 14, 2014;
we have now contributed comments to this effect.

\subsection*{Appendix 2: Special values of the polynomials  $P, Q, \Pbar, \Qbar$}

Let ${\cal P}_r$ denote any of the polynomials $P_r, Q_r, \Pbar_r, \Qbar_r$.
In Table~1 we give 
the special values ${\cal P}_r(0)$, ${\cal P}_r(1)$, and ${\cal P}_r(\infty)$,
where the latter denotes the leading coefficient of ${\cal P}_r$.

\begin{table}[h]
\begin{center}
\begin{tabular}{|c|c|c|c|c|}
\hline &&&&\\[-10pt]
${\cal P}_r$	& $P_r$	& $Q_r$	& $\Pbar_r$	& $\Qbar_r$\\[1pt]
\hline &&&&\\[-10pt]
${\cal P}_r(0)$	&
 $\delta_{0,r}$ &
  $\delta_{0,r}$ &
   $(-)^r(2r+1){\cal S}_r$ & $1$\\[1pt]
${\cal P}_r(1)$	&
 $1$	&
  $\max(1,2^{r-1})$ & $1$ &
   $(3^{2r} + 3)/4$\\[1pt]
${\cal P}_r(\infty)$	& 
 $r!$	&
  $(2r)!/(2^r r!)$ &
   $2^{2r}r!$ &
 $(2r)!/r!$\\[1pt]
\hline
\end{tabular}
\caption{Special values of the polynomials}
\end{center}
\end{table}
\noindent In Table~$1$, $\delta_{i,j}$ is $1$ if $i=j$ and $0$ otherwise;
${\cal S}_r$ is the $r$-th Secant number~\cite{rpb242}, defined by
\[
\sec x = \frac{1}{\cos x} = \sum_{r \ge 0} {\cal S}_r
           \frac{x^{2r}}{(2r)!}\,\raisedot
\]
The values
$(-)^r\Pbar_r(0)$ are OEIS sequence A009843, 
and are given by the egf
\begin{equation}		\label{eq:Pbar_gf}
\sum_{r=0}^\infty \Pbar_r(0)\frac{x^{2r+1}}{(2r+1)!}
 = \frac{x}{\cosh x}\,\raisedot
\end{equation}
They may be expressed in terms of the 
Secant numbers ${\cal S}_r$, which comprise OEIS sequence A000364.
In view of~\eqref{eq:Pbar_F}, we obtain a special value of
the Dumont-Foata polynomials: 
\begin{equation}		\label{eq:F_half_half_half}
F_{r+1}\left(\half,\half,\half\right) = 2^{-2r}(2r+1){\cal S}_r.
\end{equation}

The values $\Qbar_r(1)$ comprise OEIS sequence A054879.
The values in the last row of Table~$1$ may also be found in OEIS:
they are sequences A000142, A001147, A047053, and A001813.

Tuenter~\cite{Tuenter1} observed that, for $r \ge 1$, the constant
terms of $-P_r(n)/n$ are the
Genocchi numbers (A001469),
and the constant terms of $(-)^{r-1}Q_r(n)/n$ are the
reduced tangent numbers (A002105). 
\end{document}